\documentclass[11pt,reqno]{elsarticle}
\usepackage{fullpage}
\usepackage{dsfont}
\usepackage{amsmath}
\usepackage{amsfonts}
\usepackage{amssymb}
\usepackage{graphicx}
\usepackage{mathrsfs}
\usepackage{epsfig}
\usepackage{amsthm}
\usepackage{verbatim}
\usepackage{graphicx}
\usepackage{color}
\usepackage{url}
\usepackage{todonotes}

\usepackage[colorlinks,citecolor=black,linkcolor=black,
            bookmarksopen,
            bookmarksnumbered
           ]{hyperref}

\newcommand{\beq}{\begin{equation}}
\newcommand{\eeq}{\end{equation}}

\newcommand{\nab}{\langle \nabla \rangle_c}

\newtheorem{theorem}{Theorem}[section]

\newtheorem{lemma}[theorem]{Lemma}
\newtheorem{cor}[theorem]{Corollary}

\theoremstyle{definition}
\newtheorem{definition}{Definition}[section]

\allowdisplaybreaks[3]

\begin{document}
\begin{abstract}
In this paper we present a novel class of asymptotic consistent exponential-type integrators for Klein--Gordon--Schr\"odinger systems that capture all regimes from the slowly varying classical regime up to the highly oscillatory non-relativistic limit regime. We achieve convergence of order one and two that is uniform in $c$ without any time step size restrictions. {In particular, we establish an explicit relation between gain in negative powers of the potentially large parameter $c$ in the error constant and loss in derivative.}
\end{abstract}

\title{Uniformly accurate integrators for Klein--Gordon--Schr\"odinger systems from the classical to non-relativistic limit regime}

\author{María {Cabrera Calvo}}
\address{Laboratoire Jacques-Louis-Lions, Sorbonne Universit\'e, France}
\ead{maria.cabrera\_calvo@sorbonne-universite.fr}

\maketitle
\section{Introduction}
We consider the Klein--Gordon--Schr\"odinger system
\begin{equation}
\begin{aligned}
    \label{KGS}
    &c^{-2}\partial_{tt}z(t,x) - \Delta z(t,x) + c^2z(t,x) = \vert \psi(t,x) \vert^2, \\
    &i \partial_t \psi(t,x) + \frac12\Delta\psi(t,x) + \psi(t,x) z(t,x) = 0,
\end{aligned}
\end{equation}
given by a Klein--Gordon equation coupled nonlinearly with a classical Schr\"odinger equation. This setting arises in quantum field theory, representing the dynamics of the interaction between a complex-valued scalar nucleon field $\psi:\mathbb{R}\times\mathbb{R}^d\to\mathbb{C}$ with a neutral real-valued scalar meson field $z:\mathbb{R}\times\mathbb{R}^d\to\mathbb{R}$. For existence and uniqueness of global smooth solutions see \cite{FuTs1}, \cite{FuTs2}, \cite{FuTs3}.

The parameter $c$, proportional to the speed of light, plays a very important role in the behaviour of the solution and gives rise to two different regimes. We distinguish between the so-called \textit{relativistic regime}, where $c=1$, and the \textit{non-relativistic regime} with $c\gg 1$. The former regime is well studied numerically, for instance see \cite{BY}, as its solution is slowly varying. The non-relativistic, on the other hand, brings in a significant additional challenge in terms of its numerical treatment, given the highly oscillatory nature of its solution. It causes classical numerical methods to collapse, as they fail to capture the rapid oscillations, leading to large errors or, in turn, severe time step size restrictions and thus very intense computational efforts. This is even the case for Gautschy-type methods (see \cite{BD}, \cite{HLW}), which were specifically designed to numerically solve highly oscillatory problems. Splitting methods fail to deal with these rapid oscillations in a similar way, see for instance \cite{Fa} for their analysis in the context of Schr\"odinger equations.

In \cite{BZ}, an unconditionally stable method was developed, based on a multi-scale expansion technique, that achieves uniform linear convergence in time, for sufficiently smooth solutions. Quadratic convergence was achieved in this setting, yet only in the case where either $c=\mathcal{O}(1)$ or $c\tau\geq 1$.

In \cite{BaKoS18} an approach was presented that succeeded to capture all regimes in $c$, providing error bounds that were independent of this parameter and without requiring any step size restrictions. This was done by means of the introduction of the so-called \textit{twisted variables}, that were already well known both in physics as "interaction picture", and in the study of partial differential equations at low regularity. The main idea therein was to explicitly filter out the highly oscillatory phases, approximate the slowly varying parts, which does not produce dependency on $c$ in the error constants, and integrate the interacting highly oscillatory phases exactly. In addition, this approach achieves asymptotic consistency, meaning that it preserves the NLS limit on the discrete level.

In comparison to \cite{BaKoS18}, we propose a novel class of exponential-type integrators that equally manages to capture all regimes of $c$, without any time step size restrictions, and is asymptotically consistent. {We introduce a new construction, exploiting the structure of the leading differential operators $\frac12\Delta$ and $c\nab$, which allows us to establish an explicit relation between a gain of negative powers of the potentially very large parameter $c$ in the error constant versus a loss of derivative. In other words, a gain in accuracy, in the non-relativistic regime, in exchange of a loss in derivative. In addition to this, in the first order scheme, we require one derivative less in the Klein--Gordon part than pre-existing methods up to our knowledge.} We achieve this by employing techniques introduced in \cite{CS} in the context of the Klein--Gordon equation. This recent work couples the ideas of low regularity (in space) approximation presented in \cite{RS} in the context of an abstract class of of evolution equations, with the idea of uniform accuracy achieved in \cite{BaKoS18}.

The underlying strategy is the expansion of suitable filtered functions that allows the embedding of the full spectrum of oscillations into the numerical scheme. Here, the commutator structure of the leading operator $\frac12\Delta$ and
is studied (see Lemma \ref{lemma_comm}), as it plays a crucial role in the achievement of low regularity approximations that do not produce powers of $c$ in our error estimates.  {In addition to this, we study the asymptotic behaviour of the leading operator $c\nab$ in order to resolve the nonlinear frequency interaction caused by the coupled nature of this system.}

In the present setting, however, the fact that the equations are coupled non-linearly supposes an additional challenge. It makes the analysis much more involved, since one has to consider the non-linear interaction of highly oscillatory parts. This rises the need of new, adapted techniques.\\
\\
%In addition, we show that this approach is asymptotic consistent of order $c^{-2}$, for sufficiently smooth initial data. We carry out numerical experiments (see Figure \ref{asymptotic_cons_plots}) that illustrate our theoretical findings, and we test the sharpness of these assumptions by way of experimenting with rougher data. These experiments still hint at asymptotic convergence for less regular initial data, yet with a lower rate.\\
\noindent{\bf Outline of the paper.}
We begin by expressing \eqref{KGS} as a first order system in time in Section 2. In Section 3 we motivate the new first order scheme and it's second order counterpart will be derived in Section 4. We prove their uniform convergence in Theorems \ref{thm:1} and \ref{thm:2}, respectively. In Section 5 we briefly present the limit system, show that, as $c\to\infty$ formally, we recover the solution to the limit system. Finally, numerical experiments are presented in Section 6, confirming our theoretical results.\\
\\
\noindent{\bf Notation.}
For reasons regarding ease of implementation and clarity of presentation, we impose periodic boundary conditions, i.e. $x\in\mathbb{T}^d$. However, we note that nor the construction nor the analysis of our scheme depends on any Fourier expansion techniques, and thus can be generalised to bounded domains $x\in \Omega\subset \mathbb{R}^d$ equipped with suitable boundary conditions, as well as the full space $x\in\mathbb{R}^d$. For simplicity, we may occasionally make use of $\mathcal{O}$-notation exclusively in the context of constants independent of $c$. For the sake of simplifying future notation, we define the $\varphi_1$-function as
\begin{align}\label{phi1}
    \varphi_1(\xi) = \frac{e^{\xi}-1}{\xi},\quad \xi \in \mathbb{C}.
\end{align}
We also note 
\begin{align}\label{stability_trick}
    \varphi_1(\xi) = 1+\mathcal{O}(\xi).
\end{align}
We refer to \cite{HoOs} for details on this family of functions. In the following we fix $r>\frac{d}{2}$ and we denote by $\| . \|_r$ the standard $H^r=H^r(\mathbb{T}^d)$ Sobolev norm, where, for this choice of $r$, the following well-known bilinear estimate holds
$$
\| fg\|_r \leq C_{r,d} \|f\|_r\|g\|_r,
$$
for some constant $C_{r,d}>0$ independent of $f$ and $g$.

\section{Formulation as a first order system}
We start by defining the following differential operator that will simplify our notation significantly. For a given $c>0$ we define the following operator
\begin{align*}
    c\nab = c\sqrt{c^2-\Delta}.
\end{align*} 
One can verify that this differential operator is well-defined, as its corresponding Fourier multiplier has the form $(\nab)_k=\sqrt{k^2+c^2}$. We may now rewrite \eqref{KGS} as a first order system in time (see \cite{MaNa}). Setting
\begin{align}\label{def_u}
    u = z - ic^{-1}\nab^{-1}\partial_t z,\quad v = z - ic^{-1}\nab^{-1}\partial_t \overline{z},
\end{align}
a simple calculation shows that $z=\frac12(u+v)$. Furthermore, if we assume that $z(t,x)\in\mathbb{R}$, we have 
\begin{align}\label{z_ito_u}
    z=\frac12 (u+\overline{u}).
\end{align}
In the following we will restrict our attention to this case purely for the purpose of presenting our main ideas as clearly as possible. Note, however, that this does not impose any significant restriction. Having said this, a short calculation shows that the corresponding first order system in $(u,\psi)$ reads
\begin{align}
\label{KG}
    &i \partial_t u + c \nab u - c \nab^{-1} \vert \psi \vert ^2 =0, &&u(0)=z(0) -ic^{-1}\nab^{-1}\partial_tz(0),\\
\label{Schr}
    &i \partial_t \psi + \frac12\Delta\psi + \frac12\psi (u+\overline{u}) = 0,&&\psi(0)=\psi_0.
\end{align}
%We note that the leading operator $c\nab$ in \eqref{KG}, behaves for $c\to\infty$ as
%$$
%c\nab = c^2 + \text{lower order terms in c},
%$$
%which is a significant remark, as we aim to avoid introducing a dependency of powers of $c$ in our estimates.

\section{A first order integrator}
In this section we proceed to give a detailed derivation of the numerical scheme for $u^{n+1}\approx u(t_{n+1})$ with $t_{n+1}=t_n+\tau$, followed by a less detailed derivation of the numerical scheme for $\psi^{n+1}\approx\psi(t_{n+1})$ that employs analogous ideas. Duhamel's formula for \eqref{KG} reads
$$
u(t_n+\tau) = e^{i \tau c\nab}u(t_n) - i c \nab^{-1} e^{i \tau c\nab}\int_0^{\tau} e^{-i s c\nab} \vert \psi(t_n+s) \vert^2 ds.
$$
Iterating Duhamel's formula for \eqref{Schr} leads to
\begin{align}
\label{1o_1}
u(t_n+\tau) = e^{i \tau c\nab}u(t_n) - i c \nab^{-1} e^{i \tau c\nab}\int_0^{\tau} e^{-i s c\nab} \vert e^{i\frac12 s\Delta}\psi(t_n) \vert^2 ds + \mathcal{R}_1, 
\end{align}
where $\mathcal{R}_1$ fulfills a bound of the form
\begin{equation}\label{R_1}
    \begin{aligned}
    \|\mathcal{R}_1\|_r &\leq \bigg\| i c \nab^{-1} e^{i \tau c\nab}\int_0^{\tau} e^{-i s c\nab}\big( e^{-i\frac12s\Delta}\overline{\psi(t_n)} \big)  \mathcal{I}_{\psi}(s) \,ds \bigg\|_r \\
    &+ \bigg\| i c \nab^{-1} e^{i \tau c\nab}\int_0^{\tau} e^{-i s c\nab}\big( e^{i\frac12s\Delta}{\psi(t_n)} \big)    \mathcal{I}_{\psi}(s)  \,ds \bigg\|_r \\
    &+ \bigg\| i c \nab^{-1} e^{i \tau c\nab}\int_0^{\tau} e^{-i s c\nab}\big| \mathcal{I}_{\psi}(s)\big|^2 ds \bigg\|_r\\
    &\leq \tau^2 K\big( \sup_{0\leq\xi\leq\tau} \|\psi(t_n+\xi)\|_r\big),
    \end{aligned}
\end{equation}
where
\begin{align}\label{I_psi}
\mathcal{I}_{\psi}(s)=\frac{i}{2} e^{i\frac12 s\Delta}\int_0^{s}e^{-i\frac12\sigma\Delta}\psi(t_n+\sigma)\big( u(t_n+\sigma) + \overline{u(t_n+\sigma)} \big)d\sigma,
\end{align}
thanks to the following remarks
\begin{align}
    \label{cnab_liniso_bound}
    \|c\nab^{-1}\|_r \leq 1, \quad \|e^{itc\nab}\|_r =1,\quad \|e^{it\Delta}\|_r =1 \,\quad\forall t\in\mathbb{R}.
\end{align}

It is left to approximate the present oscillatory integral in a suitable way and, to this end, we introduce the following crucial commutator term, which will appear in our local error estimates. 
\begin{definition}
For a function $H(v_1,\dots,v_n)$, $n\geq 1$, and a linear operator $L:H^r\to H^r$ we define the following commutator type term
\begin{align*}
    \mathcal{C}[H,L] (v_1,\dots,v_n)= -L(H(v_1,\dots,v_n))+\sum_{i=1}^n D_iH(v_1,\dots,v_n)\cdot Lv_{i},
\end{align*}
where $D_iH$ stands for the partial derivative of H with respect to the variable $v_i$. Furthermore, we set
$$\mathcal{C}^2[H,L] (v_1,\dots,v_n)=  \mathcal{C}[ \mathcal{C}[H,L],L] (v_1,\dots,v_n)$$
and
$$
f_{\text{quad}}(v,w)=vw.
$$
\end{definition}
We now establish the necessary bounds for these commutator type terms.
\begin{lemma}
\label{lemma_comm}
We have that
\begin{align*}
    &\|\mathcal{C}[f_{\text{quad}}(\cdot,\cdot),\Delta] (v,w)\|_r \leq K_1 \|v\|_{r+1}\|w\|_{r+1},\\
    &\|\mathcal{C}^2[f_{\text{quad}}(\cdot,\cdot),\Delta] (v,w)\|_r \leq K_2 \|v\|_{r+2}\|w\|_{r+2}
\end{align*}
for some $K_1,K_2>0$.
\end{lemma}
\begin{proof}
We will show the first assertion in detail, the second assertion can be proven iterating this argument. By definition,
\begin{align*}
    \mathcal{C}[f_{\text{quad}}(\cdot,\cdot),\Delta] (v,w)= -\Delta (vw) + w\Delta v + v\Delta w.
\end{align*}
The assertion follows by the product rule of the Laplacian
$$
\Delta(vw) = v\Delta w + 2\nabla v \cdot\nabla w + w\Delta v.
$$
\end{proof}
We now aim to capture the oscillatory integral in \eqref{1o_1} in a way that does not trigger dependency on $c$ in the error terms and allows a low regularity approximation. The technique is captured in the following lemma.
\begin{lemma}[First order approximation of the integral in \eqref{1o_1}]
\label{Lemma1ou}
It holds that
\begin{align*}
    \int_0^{\tau} e^{-i s c\nab} \vert e^{i \frac12 s\Delta}v \vert^2 ds = \tau \overline{v}\varphi_1(i \tau (\Delta-c^2))v + \mathcal{O}\big(\tau^2(\mathcal{C}[f_{\text{quad}}(\cdot,\cdot),\Delta](v,v)+c^{-2\alpha}\Delta^{1+\alpha}v)\big).
\end{align*}
\end{lemma}
\begin{proof}
We introduce the following filtered function defined by
\begin{align}
\label{N}
    \mathcal{N}(s,s_1,\Delta,v) :=e^{i \frac12 s_1 \Delta}  \big( e^{-i\frac12 s_1\Delta}e^{i s\Delta}v \big)\big( e^{-i\frac12 s_1\Delta}\overline{v} \big).
\end{align}
Then, the integral reads 
$$
\int_0^{\tau} e^{-i s c\nab} \vert e^{i \frac12 s\Delta}v \vert^2 ds = \int_0^{\tau} e^{-is c\nab} e^{-i\frac12 s\Delta} \mathcal{N}(s,s,\Delta,v) ds.
$$
{At this point one needs to first tackle the interactions between the differential operators $c\nab$ and $\frac12\Delta$. Note that, as it can be shown via fractional Taylor series expansion of the function $x\to c^2\sqrt{c^2+x^2}$, it holds 
\begin{align}\label{taylor_cnab}
    c\nab  = c^2- \tfrac{1}{2}\Delta+ \mathcal{O}\big(\tfrac{\Delta^{1+\alpha}}{c^{2\alpha}}\big), \quad 0\leq\alpha\leq1,
\end{align}
and thus,
$$
\int_0^{\tau} e^{-is c\nab} e^{-i\frac12 s\Delta} \mathcal{N}(s,s,\Delta,v) ds = \int_0^{\tau} e^{-is c^2} \mathcal{N}(s,s,\Delta,v) ds + \mathcal{R}_1,
$$
where, using \eqref{taylor_cnab}, we see that we may bound $\mathcal{R}_1$ by
\begin{align}\label{c_trick}
    \|\mathcal{R}_1\|_r \leq c^{-2\alpha}K(\|v\|_{r+2(1+\alpha)}),\quad 0\leq \alpha \leq 1,
\end{align}
for some $K'>0$ independent of $c$.}

On the other hand, Taylor series expansion gives the following approximation
$$\mathcal{N}(s,s_1,\Delta,v)=\mathcal{N}(s,0,\Delta,v)+\int_0^{s} \partial_{s_1}\mathcal{N}(s,s_1,\Delta,v) ds_1.$$

Thus, plugging in this expansion and then integrating exactly we obtain
\begin{align*}
\int_0^{\tau} e^{-i s c\nab} e^{i\frac12 s\Delta} \mathcal{N}(s,s,\Delta,v) ds & = \int_0^{\tau} e^{-i s c^2} \mathcal{N}(s,0,\Delta,v) ds+ \mathcal{R}_1\\
& = \int_0^{\tau} e^{-i s c^2} \big( e^{i s\Delta}{v} \big)\overline{v}  ds+ \mathcal{R}_1+ \mathcal{R}_2\\
&= \tau \overline{v}\varphi_1(i \tau (\Delta-c^2))v    + \mathcal{R}_1+ \mathcal{R}_2,
\end{align*}
where $\mathcal{R}_2$ can be bounded as follows. By  \eqref{cnab_liniso_bound} and the remark
\begin{align*}
    \partial_{s_1} \mathcal{N}(s,s_1,\Delta,v)_{\vert s_1=s}  %&=i\tfrac{1}{2}\Delta e^{i \frac12 s_1 \Delta}  \big( e^{-i\frac12 s_1\Delta}e^{i s\Delta}v \big)\big( e^{-i\frac12 s_1\Delta}\overline{v} \big) - e^{i \frac12 s_1 \Delta}  \big(i\tfrac{1}{2}\Delta e^{-i\frac12 s_1\Delta}e^{i s\Delta}v \big)\big( e^{-i\frac12 s_1\Delta}\overline{v} \big) \\&-e^{i \frac12 s_1 \Delta}  \big( e^{-i\frac12 s_1\Delta}e^{i s\Delta}v \big)\big(i\tfrac{1}{2}\Delta e^{-i\frac12 s_1\Delta}\overline{v} \big)\\
    = e^{i\frac12 s_1\Delta}\mathcal{C}[f_{\text{quad}}(\cdot,\cdot),-i\tfrac{1}{2}\Delta](e^{-i\frac12 s_1\Delta}e^{is\Delta}v , e^{-i\frac12s_1\Delta}\overline{v}).
\end{align*}
it holds
\begin{align*}
    \|\mathcal{R}_2\|_r \leq \bigg\| \int_0^{\tau} e^{-i s c\nab} e^{i\frac12s\Delta} \bigg(\int_0^s\partial_{s_1} \mathcal{N}(s,s_1,\Delta,v)ds_1\bigg) ds  \bigg\|_r\leq \tau^2 K\big(\mathcal{C}[f_{\text{quad}}(\cdot,\cdot),i\tfrac{1}{2}\Delta](v , \overline{v})\big).
\end{align*}
\end{proof}

The expansion \eqref{1o_1} together with Lemma \ref{Lemma1ou} lead to the following first order uniformly accurate integrator
\begin{align}
\label{KGo1}
    u^{n+1} = e^{i \tau c\nab}u^n - i \tau c \nab^{-1} e^{i \tau c\nab}\overline{\psi^n}\varphi_1(i \tau (\Delta-c^2))\psi^n .
\end{align}

Finally, given $u^{n+1}$ by \eqref{KGo1}, one can very easily find $z^{n+1}$, via \eqref{z_ito_u}, 
$$
z^{n+1}=\frac12\big( u^{n+1} +\overline{u^{n+1}} \big).
$$

Now we proceed as follows for $\psi^{n+1}$. Duhamel's formula for \eqref{Schr} reads
$$
\psi(t_n+\tau) = e^{i\frac12\tau\Delta}\psi(t_n) + i\frac12 e^{i\frac12\tau\Delta}\int_0^{\tau}e^{-i\frac12s\Delta}\psi(t_n+s)\big( u(t_n+s) + \overline{u(t_n+s)} \big)ds.
$$
Iterating Duhamel's formula for \eqref{KG} and \eqref{Schr} respectively leads to
\begin{align}
    \label{1o_3}
    \psi(t_n+\tau) = e^{i\frac12\tau\Delta}\psi(t_n) + \frac{i}{2} e^{i\frac12\tau\Delta}\int_0^{\tau}e^{-i\frac12s\Delta}\big(e^{i \frac12s\Delta}\psi(t_n)\big)\big( e^{i sc\nab}u(t_n) + e^{-i sc\nab}\overline{u(t_n)} \big)ds+ \mathcal{R}_2,
\end{align}
where $\mathcal{R}_2$ fulfills the bound
\begin{align*}
    \| \mathcal{R}_2 \|_r \leq \|\mathcal{R}_{2,1}\|_r + \|\mathcal{R}_{2,2}\|_r, 
\end{align*}
with 
\begin{align*}
    \mathcal{R}_{2,1}=\frac{i}{2} e^{i\frac12\tau\Delta}\int_0^{\tau}e^{-i\frac12s\Delta}
   \mathcal{I}_{\psi}(s)
    \big( e^{i sc\nab}u(t_n) + e^{-i sc\nab}\overline{u(t_n)} \big)ds,
\end{align*}
with $\mathcal{I}_{\psi}$ given by \eqref{I_psi} and
\begin{align*}
    \mathcal{R}_{2,2}=\frac{i}{2} e^{i\frac12\tau\Delta}\int_0^{\tau}e^{-i\frac12s\Delta}\big(e^{i \frac12s\Delta}\psi(t_n)\big)\big( \mathcal{I}_u(s)+ \overline{\mathcal{I}_u(s)} \big)ds,
\end{align*}
\begin{align}\label{I_u}
\mathcal{I}_u(s)=i c \nab^{-1} e^{i s c\nab}\int_0^{s} e^{-i \sigma c\nab} \vert \psi(t_n+\sigma) \vert^2 d\sigma.
\end{align}
Thus, we conclude by \eqref{cnab_liniso_bound} that
\begin{align}
    \label{R_2}
    \|\mathcal{R}_2\|_r \leq \tau^2 K\big(\sup_{t_n\leq t\leq t_{n+1}} \|u(t)\|_r,\sup_{t_n\leq t\leq t_{n+1}} \|\psi(t)\|_r\big).
\end{align}

It is left to approximate the highly oscillatory integral in \eqref{1o_3} and to this end we proceed analogously as in Lemma \ref{Lemma1ou}.
\begin{lemma}[First order approximation of the integral in \eqref{1o_3}]
\label{Lemma1opsi}
It holds that
\begin{align*}
    \int_0^{\tau}e^{-i\frac12s\Delta}&\big(e^{i \frac12s\Delta}v\big)\big( e^{i sc\nab}w + e^{-i sc\nab}\overline{w} \big)ds \\
    &=\tau v\big( \varphi_1(i\tau(c\nab-\tfrac{1}{2}\Delta))w + \varphi_1(-i\tau(c\nab+\tfrac{1}{2}\Delta))\overline{w} \big) + \mathcal{O}(\tau^2 \mathcal{C}[f_{\text{quad}}(\cdot,\cdot),i\Delta](v, w)).
\end{align*}
\end{lemma}
\begin{proof}
We define the following filtered functions
$$
\mathcal{N}(s,s_1,\Delta,v,w)= e^{-i\frac12s_1\Delta}\big(e^{i\frac12 s_1\Delta}v)(e^{i\frac12 s_1\Delta}e^{-i \frac12s\Delta} (e^{i sc\nab}w + e^{-i sc\nab}\overline{w})\big).
$$
Taylor series expansion around the point $s_1=0$
%$$
%\mathcal{N}(s,s_1,\Delta,v,w) = \mathcal{N}(s,0,\Delta,v,w) + \int_0^s \partial_{s_1}\mathcal{N}(s,s_1,\Delta,v,w) ds_1
%$$
yields the following first order approximation
\begin{align*}
    \int_0^{\tau}e^{-i\frac12s\Delta}\big(e^{i \frac12s\Delta}v\big)&\big( e^{i sc\nab}w + e^{-i sc\nab}\overline{w} \big)ds\\
    &=  \int_0^{\tau}\mathcal{N}(s,s,\Delta,v,w)ds =\int_0^{\tau}\mathcal{N}(s,0,\Delta,v,w)ds+ \mathcal{R}
    \\&={\tau v\big( \varphi_1(i\tau(c\nab-\tfrac{1}{2}\Delta))w + \varphi_1(-i\tau(c\nab+\tfrac{1}{2}\Delta))\overline{w} \big)} + \mathcal{R}.
\end{align*}
Once again we find a bound for $\mathcal{R}$ thanks to the observation
$$
\partial_{s_1} \mathcal{N}(s,s_1,\Delta,v,w)=e^{-i\frac12s_1\Delta}\mathcal{C}[f_{\text{quad}}(\cdot,\cdot),i\tfrac{1}{2}\Delta](e^{i\frac12s_1\Delta}v, e^{i\frac12s_1\Delta}(e^{is_1c\nab}w+e^{-is_1c\nab}\overline{w}  ) ),
$$
thus
$$
\|\mathcal{R}\|_r\leq \tau^2 K\big( \mathcal{C}[f_{\text{quad}}(\cdot,\cdot),i\tfrac{1}{2}\Delta](v , w)\big).
$$
\end{proof}

Plugging our findings from Lemmata \ref{Lemma1ou} and \ref{Lemma1opsi} into the expansion \eqref{1o_3} motivates the following scheme for $\psi^{n+1}$.
\begin{align}
    \label{Schro1}
    \psi^{n+1} = e^{i \frac12\tau\Delta}\psi^{n} + \tau \frac{i}{2} e^{i\frac12\tau\Delta}\psi^{n}\big( \varphi_1(i\tau(c\nab-\tfrac{1}{2}\Delta))u^{n} + \varphi_1(-i\tau(c\nab+\tfrac{1}{2}\Delta))\overline{u^{n}} \big).
\end{align}
In the sections that follow we aim to carry out the error analysis of the scheme in $(u^{n},\psi^{n})$ given by \eqref{KGo1} and \eqref{Schro1}. Before we begin we denote by $\varphi^t_{K}$, $\varphi^t_{S}$ the exact flows of \eqref{KG} and \eqref{Schr} respectively and by $\Phi^t_{K}$, $\Phi^t_{S}$ the numerical flows corresponding to \eqref{KGo1} and \eqref{Schro1} respectively, such that in particular it holds
$$
u(t_{n+1})=\varphi^{\tau}_K(u(t_n),\psi(t_n)),\, \psi(t_{n+1})=\varphi^{\tau}_S(u(t_n),\psi(t_n)),\quad u^{n+1}=\Phi^{\tau}_{K}(u^n,\psi^n),\, \psi^{n+1}=\Phi^{\tau}_{S}(u^n,\psi^n).
$$
\subsection{Local error analysis}
\begin{lemma}
\label{localerror1}
Fix $r>\frac{d}{2}$. The local error given by the differences $\varphi^{\tau}_K(u(t_n),\psi(t_n))-\Phi^{\tau}_K(u(t_n),\psi(t_n))$ and $\varphi^{\tau}_S(u(t_n),\psi(t_n))-\Phi^{\tau}_S(u(t_n),\psi(t_n))$ satisfies
$$
\varphi^{\tau}_K(u(t_n),\psi(t_n))-\Phi^{\tau}_K(u(t_n),\psi(t_n)) = \mathcal{O}(\tau^2(\mathcal{C}[f_{\text{quad}}(\cdot,\cdot),i\tfrac{1}{2}\Delta]+c^{-2\alpha}\Delta^{1+\alpha}))(\psi(t_n) ,\overline{\psi(t_n)}))
$$
and
$$
\varphi^{\tau}_S(u(t_n),\psi(t_n))-\Phi^{\tau}_S(u(t_n),\psi(t_n)) = \mathcal{O}(\tau^2 \mathcal{C}[f_{\text{quad}}(\cdot,\cdot),i\tfrac{1}{2}\Delta](\psi(t_n) , w)),
$$
for $w\in \{ u(t_n),\overline{u(t_n)} \}$.
\end{lemma}
\begin{proof}
This assertion follows by the bounds we have found for the remainder terms in \eqref{R_1} and \eqref{R_2}, Lemma \ref{Lemma1ou} and Lemma \ref{Lemma1opsi}, together with Lemma \ref{lemma_comm}.
\end{proof}
\subsection{Stability analysis}
\begin{lemma}
\label{stability1}
Fix $r>\frac{d}{2}$. The numerical flows $\Phi^{\tau}_K$ and $\Phi^{\tau}_S$ defined by \eqref{KGo1} and \eqref{Schro1} respectively are stable in $H^r$, namely it holds for any $v_i,w_i\in H^r$, $i\in\{1,2\}$ that
$$
\|\Phi^{\tau}_K(v_1,w_1)-\Phi^{\tau}_K(v_2,w_2)\|_r\leq \|v_1-v_2\|_r + \tau M_1 \|w_1-w_2\|_r,
$$
$$
\|\Phi^{\tau}_S(v_1,w_1)-\Phi^{\tau}_S(v_2,w_2)\|_r\leq \|w_1-w_2\|_r + \tau M_2 (\|v_1-v_2\|_r +\|w_1-w_2\|_r),
$$
where $M_1$ and $M_2$ can be chosen independently of $c$.
\end{lemma}
\begin{proof}
This claim follows by \eqref{cnab_liniso_bound}. In addition, we use the estimate $\vert \varphi_1 (i\xi)\vert\leq 1$ for all $\xi\in\mathbb{R}$.
\end{proof}
\subsection{Global error}
\begin{theorem}\label{thm:1}
Fix $r>\frac{d}{2}$ and assume that the solution $(u,\psi)$ of \eqref{KG}-\eqref{Schr} satisfies $u\in\mathcal{C}([0,T],H^{r+1})$, {$\psi\in\mathcal{C}([0,T],H^{r+2(1+\alpha)})$, $0\leq\alpha\leq1$}. Then there exists a $\tau_0>0$ such that for all $0<\tau\leq\tau_0$ the following estimate holds for $(u^n,\psi^n)$ defined in \eqref{KGo1} and \eqref{Schro1}
\begin{align*}
    \|u(t_n)-u^n\|_r + \|\psi(t_n)-\psi^n\|_r &\leq \tau K_1\big(\sup_{t_n\leq t\leq t_{n+1}} \|u(t)\|_{r+1},\sup_{t_n\leq t\leq t_{n+1}} \|\psi(t)\|_{r+1}\big)\\
&{+\tau c^{-2\alpha}K_2\big(\sup_{t_n\leq t\leq t_{n+1}}\|\psi(t)\|_{+2(1+\alpha)}\big)},
\end{align*}

where, in particular, $K_1$ and $K_2$ can be chosen independently of $c$.
\end{theorem}
\begin{proof}
The proof follows by means of a Lady Windermere's fan argument (see, for example \cite{HNW}), after plugging in the results obtained in Lemmata \ref{localerror1} and \ref{stability1}.
\end{proof}

{\noindent{\bf Remark.} We note that, in the fully discrete case of the non-relativistic with highest discrete frequency $\vert K \vert \ll c^{-2\alpha}$, the second term in the global error estimate presented above becomes negligible, as the contribution of the higher Sobolev norm is nearly cancelled by the very small parameter $c^{-2\alpha}$. In this case the error constant is then lead by the $H^{r+1}$ norm of the solution, thus allowing for lower regularity assumptions than in pre-exisiting methods in practice.}

\section{A second order integrator}
We dedicate this section to the derivation of a second order counterpart of the uniformly accurate low regularity integrator we have obtained in the previous section. Iterating Duhamel's formula for \eqref{Schr} yields
\begin{equation}\label{2o_1}
    \begin{aligned}
            u(t_n+\tau) &= e^{i \tau c\nab}u(t_n) - i c \nab^{-1} e^{i \tau c\nab}\int_0^{\tau} e^{-i s c\nab} \vert e^{is\frac12\Delta}\psi(t_n) \vert^2 ds \\
    &-ic\nab^{-1}e^{i \tau c\nab}\int_0^{\tau}e^{-i s c\nab}\big(  e^{-is\frac12\Delta}\overline{\psi(t_n)}\big)\mathcal{I}_{\psi}(s)ds\\
    &-ic\nab^{-1}e^{i \tau c\nab}\int_0^{\tau}e^{-i s c\nab}\big(  e^{is\frac12\Delta}\psi(t_n)\big)\overline{\mathcal{I}_{\psi}(s)}  ds + \mathcal{R}'_3,
    \end{aligned}
\end{equation}
where, if we once again iterate Duhamel's formula, we obtain
\begin{equation}\label{Ipsitilde}
    \begin{aligned}
    \mathcal{I}_{\psi}(s)&=i\frac12 e^{i\frac12s\Delta }\int_0^se^{-i\frac12\sigma\Delta}\psi(t_n+\sigma)\big(u(t_n+\sigma)+\overline{u(t_n+\sigma)}\big)d\sigma\\
    &=i\frac12 e^{i\frac12s\Delta }\int_0^se^{-i\frac12\sigma\Delta}\big(e^{i\frac12\sigma\Delta}\psi(t_n)\big)\big(e^{i\sigma c\nab}u(t_n)+e^{-i\sigma c\nab}\overline{u(t_n)}\big)d\sigma +\mathcal{R}''_3\\
    &=\Tilde{\mathcal{I}}_{\psi}(s)+\mathcal{R}''_3.
    \end{aligned}
\end{equation}
For the remainder terms $\mathcal{R}'_3$ and $\mathcal{R}''_3$, using \eqref{cnab_liniso_bound}, we find that
\begin{equation}
    \begin{aligned}
    \label{2o_rt1}
    \|\mathcal{R}'_3\|_r &\leq \bigg\| ic\nab^{-1}\int_0^{\tau}e^{-i s c\nab}\vert\mathcal{I}_{\psi}(s)\vert^2  ds \bigg\|_r\\
    &\leq \tau^3 K\big( \sup_{0\leq\xi\leq\tau}\|u(t_n+\xi)\|_r,\sup_{0\leq\xi\leq\tau}\|\psi(t_n+\xi)\|_r \big),
    \end{aligned}
\end{equation}
and, similarly,
\begin{equation}\label{2o_rtt1}
\begin{aligned}
   \|\mathcal{R}''_3\|_r &\leq \bigg\| e^{i\frac12s\Delta }\int_0^se^{-i\frac12\sigma\Delta} \mathcal{I}_{\psi}(\sigma) \big(e^{i\sigma c\nab}u(t_n)+e^{-i\sigma c\nab}\overline{u(t_n)}\big)d\sigma \bigg\|_r\\
    &+\bigg\| e^{i\frac12s\Delta }\int_0^se^{-i\frac12\sigma\Delta}\big(e^{i\frac12\sigma\Delta}\psi(t_n)\big)\big(\mathcal{I}_u(\sigma) + \overline{\mathcal{I}_u(\sigma)}\big)d\sigma \bigg\|_r \\
    &+ \bigg\| e^{i\frac12s\Delta }\int_0^se^{-i\frac12\sigma\Delta} \mathcal{I}_{\psi}(\sigma) \big(\mathcal{I}_u(\sigma) + \overline{\mathcal{I}_u(\sigma)}\big)d\sigma \bigg\|_r \\
    & \leq s^2 K\big( \sup_{0\leq\xi\leq\tau}\|u(t_n+\xi)\|_r,\sup_{0\leq\xi\leq\tau}\|\psi(t_n+\xi)\|_r \big),
\end{aligned}
\end{equation}
where $\mathcal{I}_u$ is given in \eqref{I_u}. Now, after these considerations, the expansion given in \eqref{2o_1} reads
\begin{equation}\label{2o_2}
    \begin{aligned}
            u(t_n+\tau) &= e^{i \tau c\nab}u(t_n) - i c \nab^{-1} e^{i \tau c\nab}\mathfrak{I}_u(u(t_n),\psi(t_n)) + \mathcal{R}_3,
    \end{aligned}
\end{equation}
where
\begin{align}
\label{I1}
    \mathfrak{I}_u(u(t_n),\psi(t_n)) &= \int_0^{\tau} e^{-i s c\nab} \vert e^{i\frac12s\Delta}\psi(t_n) \vert^2 ds \\
    \label{I2}
    &+\int_0^{\tau}e^{-i s c\nab}\big(  e^{-i\frac12s\Delta}\overline{\psi(t_n)}\big)\Tilde{\mathcal{I}}_{\psi}(s)ds\\
    \label{I3}
    &+\int_0^{\tau}e^{-i s c\nab}\big(  e^{i\frac12s\Delta}\psi(t_n)\big)\overline{\Tilde{\mathcal{I}}_{\psi}(s)}  ds ,
\end{align}
recall the definition of $\tilde{\mathcal{I}}_{\psi}$ is given in \eqref{Ipsitilde}. Using \eqref{2o_rt1} and \eqref{2o_rtt1}, we see that
\begin{align}
\label{2o_r1}
    \|\mathcal{R}_3\|_r\leq \tau^3 K\big( \sup_{0\leq\xi\leq\tau}\|u(t_n+\xi)\|_r,\sup_{0\leq\xi\leq\tau}\|\psi(t_n+\xi)\|_r \big),
\end{align}
where $K>0$ is chosen independently of $c$.

Prior to treating the three highly oscillatory integrals separately, we define the function
\begin{align}\label{psi2}
    \Psi_2(\xi)= \frac{e^{\xi}-\varphi_1(\xi)}{\xi},
\end{align}
recall $\varphi_1$ is given in \eqref{phi1}. We again refer to \cite{HoOs} for details on this family of functions.
\begin{lemma}[Second order approximation of the integral \eqref{I1}]
\label{Lemma2ou1}
For $0\leq\alpha\leq1$ it holds that 
\begin{align*}
    I_1(w,v) &=\int_0^{\tau} e^{-i s c\nab} \vert e^{i\frac12s\Delta}v \vert^2 ds\\
    &= \tau \big[\overline{v}\varphi_1(i \tau (\Delta-c^2)){v} +  e^{i\frac12\tau\Delta} \big( e^{-i\frac12\tau\Delta} \Psi_2(i \tau (\Delta-c^2)) v \big)\big(e^{-i\frac12\tau\Delta}\overline{v}\big)  - \overline{v}\Psi_2(i \tau (\Delta-c^2)) {v}  \big] \\
    &- i\tau^2 \varphi_1(i\tau(c\nab-c^2+\frac12\Delta))(c\nab-c^2+\frac12\Delta) \big(\overline{v}\Psi_2(i\tau(c\nab-c^2+\frac12\Delta))v\big)\\
    &+  \mathcal{O}{\big(\tau^3(\mathcal{C}^2[f_{\text{quad}}(\cdot,\cdot),i\Delta]( v, \overline{v})+c^{-4\alpha}\Delta^{2+2\alpha} \mathcal{C}[f_{\text{quad}}(\cdot,\cdot),i\Delta]( v, \overline{v}) \big)}\\
    &=\tilde{I}_1(w,v)+ \mathcal{O}{\big(\tau^3(\mathcal{C}^2[f_{\text{quad}}(\cdot,\cdot),i\Delta]( v, \overline{v})+c^{-4\alpha}\Delta^{2+2\alpha} \mathcal{C}[f_{\text{quad}}(\cdot,\cdot),i\Delta]( v, \overline{v}) \big)}.
\end{align*}
\end{lemma}
\begin{proof}
For $\mathcal{N}$ defined in \eqref{N} we see that
\begin{equation}\label{2o_i1}
    \begin{aligned}
           I_1(v)&= \int_0^{\tau} e^{-i s c\nab} e^{-i\frac12 s \Delta} \mathcal{N}(s,s,\Delta,v)  ds\\
           &{=\int_0^{\tau} e^{-i s c^2} e^{-i\frac12 s (c\nab-c^2+\frac12\Delta)} \mathcal{N}(s,s,\Delta,v)  ds}\\
           &{= \int_0^{\tau} e^{-i s c^2} (1-is(c\nab-c^2+\frac12\Delta)) \mathcal{N}(s,s,\Delta,v)  ds + \mathcal{R}_1}\\
           &{= \int_0^{\tau} e^{-i s c^2} \mathcal{N}(s,s,\Delta,v)  ds - \int_0^{\tau} e^{-i s c^2} is(c\nab-c^2+\frac12\Delta)\mathcal{N}(s,s,\Delta,v)  ds +\mathcal{R}_1},\\
    \end{aligned}
\end{equation}
where, by \eqref{c_trick}, it holds
$$
\|\mathcal{R}_1\|_r \leq \tau^3 K \big(c^{-4\alpha}\Delta^{2+2\alpha}\vert v \vert^2 \big).
$$
We tackle the first integral in \eqref{2o_i1}. A second order expansion of $\mathcal{N}$ reads
$$
\mathcal{N}(s,s,\Delta,v)=\mathcal{N}(s,0,\Delta,v)+s\partial_{s_1}\mathcal{N}(s,s_1,\Delta,v)_{s_1=0} + \int_0^s \int_0^{\sigma} \partial^2_{s_1}\mathcal{N}(s,s_1,\Delta,v) \,ds_1\, d\sigma,
$$
where $\partial^2_{s_1}\mathcal{N}(s,s_1,\Delta,v)$ obeys 
\begin{align*}
    \partial^2_{s_1}\mathcal{N}(s,s_1,\Delta,v) =e^{i\frac12 s_1\Delta}\mathcal{C}^2[f_{\text{quad}}(\cdot,\cdot),-i\tfrac{1}{2}\Delta](e^{-i\frac12 s_1\Delta}e^{is\Delta}v , e^{-i\frac12s_1\Delta}\overline{v}),
\end{align*}
recall the notation
$$
\mathcal{C}^2[f_{\text{quad}}(\cdot,\cdot),i\tfrac{1}{2}\Delta](v , w) = \mathcal{C}[\mathcal{C}[f_{\text{quad}}(\cdot,\cdot),i\tfrac{1}{2}\Delta],i\tfrac{1}{2}\Delta](v , w).
$$
In order to guarantee the stability of this scheme, we employ a finite difference approximation of $\partial_{s_1}\mathcal{N}(s,0,\Delta,v)$, namely, for $0\leq \tau\leq s$,
$$
\partial_{s_1}\mathcal{N}(s,0,\Delta,v) = \frac{\mathcal{N}(s,\tau,\Delta,v)-\mathcal{N}(s,0,\Delta,v)}{\tau} + \mathcal{O}(\tau\partial^2_{s_1}\mathcal{N}(s,s_1,\Delta,v)).
$$
With this in mind and using definition \eqref{psi2}, we see that \eqref{2o_i1} reads
\begin{equation}\label{intu11}
    \begin{aligned}
    \int_0^{\tau} e^{-i s c^2} \mathcal{N}(s,s,\Delta,v)  ds &=\int_0^{\tau} e^{-i s c^2} \bigg(\mathcal{N}(s,0,\Delta,v)+\frac{s}{\tau}\big( \mathcal{N}(s,\tau,\Delta,v)-\mathcal{N}(s,0,\Delta,v) \big) \bigg) ds+ \mathcal{R}_1+ \mathcal{R}_2
    %\\&=\int_0^{\tau} e^{-i s c^2}(1-is(c\nab-c^2+\Delta)) \bigg(\overline{v}e^{is\Delta}{v}+\frac{s}{\tau}\big( e^{i\frac12\tau\Delta} \big( e^{-i\frac12\tau\Delta} e^{is\Delta} v \big)\big(e^{-i\frac12\tau\Delta}\overline{v}\big) -\overline{v}e^{is\Delta}{v} \big) \bigg) ds \\
    %&+ \mathcal{R}_1+ \mathcal{R}_2
    \\&= \tau \big[\overline{v}\varphi_1(i \tau (\Delta-c^2)){v} +  e^{i\frac12\tau\Delta} \big( e^{-i\frac12\tau\Delta} \Psi_2(i \tau (\Delta-c^2)) v \big)\big(e^{-i\frac12\tau\Delta}\overline{v}\big)  \\
    &- \overline{v}\Psi_2(i \tau (\Delta-c^2)) {v}  \big] + \mathcal{R}_1+ \mathcal{R}_2,
    \end{aligned}
\end{equation}
where $\mathcal{R}_2$ satisfies
\begin{align*}
    \|\mathcal{R}_2\|_r &\leq \tau^3 K\big( \mathcal{C}^2[f_{\text{quad}}(\cdot,\cdot),i\Delta]( v, \overline{v}) \big).
\end{align*}
As for the second integral in \eqref{2o_i1}, we note that it suffices to carry out a Taylor expansion up to first order in $s$. We obtain
\begin{equation}\label{intu12}
    \begin{aligned}
    &\int_0^{\tau} e^{-i s c^2} is(c\nab-c^2+\frac12\Delta)\mathcal{N}(s,s,\Delta,v)  ds\\ &=\int_0^{\tau} e^{-i s c^2} is(c\nab-c^2+\frac12\Delta)\mathcal{N}(s,0,\Delta,v)  ds + \mathcal{R}_3\\
    &=i\tau^2 (c\nab-c^2+\frac12\Delta) \big(\overline{v}\Psi_2(i\tau(c\nab-c^2+\frac12\Delta))v\big)+ \mathcal{R}_3\\
    &=i\tau^2 \varphi_1(i\tau(c\nab-c^2+\frac12\Delta))(c\nab-c^2+\frac12\Delta) \big(\overline{v}\Psi_2(i\tau(c\nab-c^2+\frac12\Delta))v\big)+ \mathcal{R}_3+ \mathcal{R}_4
    \end{aligned}
\end{equation}
where, given \eqref{c_trick} and the above first order Taylor expansion of $N(s,s,\Delta,v)$, $\mathcal{R}_3$ satisfies
$$
\|\mathcal{R}_3 \|_r\leq \tau^3 K\big( c^{-2\alpha}\Delta^{1+\alpha} \mathcal{C}[f_{\text{quad}}(\cdot,\cdot),i\Delta]( v, \overline{v}) \big).
$$
We note that in the last step of \eqref{intu12} we have introduced the factor $\varphi_1(i\tau(c\nab-c^2+\frac12\Delta))$ in order to ensure stability and, by \eqref{stability_trick}, \eqref{c_trick} and the estimate obtained for $\mathcal{R}_3$, $\mathcal{R}_4$ satisfies
$$
\|\mathcal{R}_4 \|_r\leq \tau^3 K\big( c^{-4\alpha}\Delta^{2+2\alpha} \mathcal{C}[f_{\text{quad}}(\cdot,\cdot),i\Delta]( v, \overline{v}) \big).
$$
Using \eqref{intu11} and \eqref{intu12} we achieve the desired second order approximation of \eqref{2o_i1}.
\end{proof}

\begin{lemma}[Second order approximation of the integral \eqref{I2}]
\label{Lemma2ou2}
For $0\leq\alpha\leq1$ it holds that
\begin{align*}
    I_2(w,v)&=\int_0^{\tau}e^{-i s c\nab}\big(  e^{-i\frac12s\Delta}\overline{v}\big)\Tilde{\mathcal{I}}_{\psi}(s)ds \\
    &= \frac{\tau}{2c^2}\overline{v} \big[ (\varphi_1(it(\Delta+2c^2))-\varphi_1(it(\Delta+c^2)))vw - (\varphi_1(it\Delta)-\varphi_1(it(\Delta+c^2)))v\overline{w}  \big]
    \\&+\mathcal{O}{\big(\tau^3(\mathcal{C}[f_{\text{quad}}(\cdot,\cdot),i\Delta](v , w)+\Delta w + c^{-2\alpha}\Delta^{1+\alpha}(vw) )\big)}\\
    &=\tilde{I}_2(w,v) +\mathcal{O}{\big(\tau^3(\mathcal{C}[f_{\text{quad}}(\cdot,\cdot),i\Delta](v , w)+\Delta w + c^{-2\alpha}\Delta^{1+\alpha}(vw))\big)}.
\end{align*}
\end{lemma}
\begin{proof}
Applying the result found in \eqref{Schro1} within $\tilde{\mathcal{I}}_{\psi}(s)$, we find that
\begin{align*}
    I_2(w,v)&=\int_0^{\tau}e^{-i s c\nab}\big(  e^{-i\frac12s\Delta}\overline{v}\big)\Tilde{\mathcal{I}}_{\psi}(s)ds\\
    &=i \frac12 \int_0^{\tau}e^{-i s c\nab}\big(  e^{-i\frac12s\Delta}\overline{v}\big) \big( se^{i\frac12s\Delta}v(\varphi_1(is(c\nab-\tfrac{1}{2}\Delta))w + \varphi_1(-is(c\nab +\tfrac{1}{2}\Delta))\overline{w}) \big)  ds\\& +\mathcal{R}_1,
\end{align*} 
where, by Lemma \ref{Lemma1opsi} and with \eqref{cnab_liniso_bound}, it holds for $\mathcal{R}_1$ that

\begin{equation*}
    \begin{aligned}
        \| \mathcal{R}_1\|_r \leq \tau^3 K\big(\mathcal{C}[f_{\text{quad}}(\cdot,\cdot),i\tfrac{1}{2}\Delta](v, w)\big).
    \end{aligned}
\end{equation*}

Note that, by \eqref{c_trick} it holds formally that
\begin{equation}\label{anastasiya}
    \begin{aligned}
    s\varphi_1(is(c\nab-\tfrac{1}{2}\Delta)) &= \int_0^s e^{i\sigma(c\nab-\tfrac{1}{2}\Delta)}d\sigma\\
    & =s\varphi_1(isc^2)+ s^2\mathcal{O}(\Delta + c^{-2\alpha}\Delta^{1+\alpha}). 
    \end{aligned}
\end{equation}
Thus,
\begin{equation*}
    \begin{aligned}
    I_2(w,v)=i\frac12\int_0^{\tau}e^{-i s c\nab}\big(  e^{-i\frac12s\Delta}\overline{v}\big) \big( se^{i\frac12s\Delta}v (\varphi_1(isc^2)w+ \varphi_1(-isc^2)\overline{w}) \big)  ds+\mathcal{R}_1 +\mathcal{R}_2,
\end{aligned}
\end{equation*}
and, by \eqref{anastasiya}, $\mathcal{R}_2$ fulfills the bound 
$$
\| \mathcal{R}_2\|_r \leq \tau^3 K(v(\Delta + c^{-2\alpha}\Delta^{1+\alpha})w).
$$
We now let
$$
\mathcal{N}(s,s_1,v,w) = e^{i\frac12s_1\Delta}\big(  e^{-i\frac12s_1\Delta}\overline{v}\big) \big( e^{is\Delta}e^{-i\frac12s_1\Delta}v (\varphi_1(isc^2)w+ \varphi_1(-isc^2)\overline{w}) \big),
$$
and the assertion is obtained following the same line of argumentation as in Lemma \ref{Lemma2ou1}.
%obtaining
%\begin{align*}
%    I_2(w,v)&=i\frac12\int_0^{\tau}e^{-i s (c\nab+\frac12\Delta)}\mathcal{N}(s,s,v,w)  ds +\mathcal{R}_1 +\mathcal{R}_2\\
%    &=i\frac12\int_0^{\tau}se^{-i s c^2}\mathcal{N}(s,s,v,w) ds+\mathcal{R}_1 +\mathcal{R}_2 + \mathcal{R}_3,
%\end{align*}
%with
%$$
%\|\mathcal{R}_3\|_r\leq t^3 K\big( c^{-2\alpha}\Delta^{1+\alpha}(vw) \big),
%$$
%by \eqref{c_trick}. Now, it suffices to carry out Taylor series expansion of %$\mathcal{N}(s,s_1,v,w)$ up to order one, which yields
%\begin{align*}
%    I_2(w,v)&=i\frac12\int_0^{\tau}se^{-i s c^2}\mathcal{N}(s,0,v,w) ds+\mathcal{R}_1 +\mathcal{R}_2 + \mathcal{R}_3+ \mathcal{R}_4\\
%    &=i\frac12\int_0^{\tau}se^{-i s c^2} \overline{v} \big( e^{is\Delta}v (\varphi_1(isc^2)w+ \varphi_1(-isc^2)\overline{w}) \big)  ds+\mathcal{R}_1 +\mathcal{R}_2 + \mathcal{R}_3+ \mathcal{R}_4\\
%    &=\frac{\tau}{2c^2}\overline{v} \big[ (\varphi_1(i\tau(\Delta+2c^2))-\varphi_1(i\tau(\Delta+c^2)))vw - (\varphi_1(i\tau\Delta)-\varphi_1(i\tau(\Delta+c^2)))v\overline{w}  \big]  \\
%    &+\mathcal{R}_1 +\mathcal{R}_2 + \mathcal{R}_3 + \mathcal{R}_4,
%\end{align*}
%where 
%$$
%\|\mathcal{R}_4\|_r\leq t^3 K\big(\mathcal{C}[f_{\text{quad}}(\cdot,\cdot),i\tfrac{1}{2}\Delta](v, w)\big).
%$$
\end{proof}

As for the third integral, following analogous steps as the ones that lead to the approximation of \eqref{I2}, we obtain the following second order approximation of \eqref{I3} for $0\leq\alpha\leq1$
\begin{equation}
\label{2o_i3_final}
    \begin{aligned}
    {I_3(w,v)} &= \int_0^{\tau}e^{-i s c\nab}\big(  e^{i\frac12s\Delta}v\big)\bigg( -i\frac12 e^{-i\frac12s\Delta }\int_0^se^{i\frac12\sigma\Delta}\big(e^{-i\frac12\sigma\Delta}\overline{v}\big)\big(e^{i\sigma c\nab}w+e^{-i\sigma c\nab}\overline{w}\big)d\sigma \bigg)  ds \\
    &= \frac{\tau}{2c^2} \big[-\overline{v}w (\varphi_1(i\tau\Delta)-\varphi_1(i\tau(\Delta-c^2)))v+\overline{v}\overline{w}(\varphi_1(i\tau(\Delta-2c^2))-\varphi_1(i\tau(\Delta-c^2)))v\big]\\
    &+ \mathcal{O}{\big(\tau^3(\mathcal{C}[f_{\text{quad}}(\cdot,\cdot),i\Delta](v , w)+\Delta w + c^{-2\alpha}\Delta^{1+\alpha}(vw))\big)}\\
    &= \tilde{I}_3(w,v) +\mathcal{O}{\big(\tau^3(\mathcal{C}[f_{\text{quad}}(\cdot,\cdot),i\Delta](v , w)+\Delta w + c^{-2\alpha}\Delta^{1+\alpha}(vw))\big)}.
    \end{aligned}
\end{equation}

% ---------------------------------
Collecting our results from Lemma \ref{Lemma2ou1}, Lemma \ref{Lemma2ou2} and \eqref{2o_i3_final}, together with the bound proven in Lemma \ref{lemma_comm}, yields the following second order approximation of the oscillatory integral $\mathfrak{I}_u(w,v)$ in \eqref{2o_2}.
\begin{cor}\label{cor:I_u}
For $0\leq\alpha\leq1$ it holds that
\begin{equation*}
    \begin{aligned}
            \mathfrak{I}_u(w,v) &= \tilde{I}_1(w,v)+\tilde{I}_2(w,v)+\tilde{I}_3(w,v)+\mathcal{O}{ (\tau^3(\Delta w)(\Delta v))}\\
            &= \tilde{\mathfrak{I}}_u(w,v)+\mathcal{O}{\big(\tau^3((\Delta w)(\Delta v) + c^{-2\alpha}\Delta^{1+\alpha}(vw) + c^{-4\alpha}\Delta^{2+2\alpha}(vw))\big)},
    \end{aligned}
\end{equation*}
\end{cor}

% ---------------------------------

Finally, Corollary \ref{cor:I_u} leads to the following second order approximation based on \eqref{2o_2}:
\begin{equation}\label{KGo2}
    \begin{aligned}
            u^{n+1}=e^{i \tau c\nab}u^n-ic\nab^{-1}e^{i \tau c\nab}\tilde{\mathfrak{I}}_u(u^n,\psi^n).
    \end{aligned}
\end{equation}

We may now consider Duhamel's formula for \eqref{Schr}, where we iterate Duhamel's formula for \eqref{KG} and \eqref{Schr} respectively.
\begin{align*}
    \psi(t_n+\tau) &= e^{i\frac12\tau\Delta}\psi(t_n) + i\frac12 e^{i\frac12\tau\Delta}\mathfrak{I}_{\psi}(u(t_n),\psi(t_n)) + \mathcal{R}_4,
\end{align*}
where $\mathcal{R}_4$ fulfills 
\begin{align}\label{R_4}
    \|\mathcal{R}_4\|_r \leq \bigg\| \int_0^{\tau}e^{-i\frac12s\Delta}\mathcal{I}_{\psi}(s)(\mathcal{I}_u(s) + \overline{\mathcal{I}_u(s)})ds \bigg\|_r
\leq \tau^3  K\big(\sup_{0\leq \xi\leq \tau} \|u(t_n+\xi)\|_r,\sup_{0\leq \xi\leq \tau} \|\psi(t_n+\xi)\|_r\big),
\end{align}
and the integral $\mathfrak{I}_{\psi}$ reads
\begin{align}
    \mathfrak{I}_{\psi}(w,v)&= \int_0^{\tau}e^{-i\frac12s\Delta}\big( e^{i\frac12s\Delta}\psi(t_n) \big)\big( e^{isc\nab}u(t_n) + e^{-isc\nab}\overline{u(t_n)} \big)ds\label{J1}\\
    &+\int_0^{\tau}e^{-i\frac12s\Delta}\mathcal{I}_{\psi}(s)\big( e^{isc\nab}u(t_n) + e^{-isc\nab}\overline{u(t_n)} \big)ds\label{J2}\\
    &+\int_0^{\tau}e^{-i\frac12s\Delta}\big( e^{i\frac12s\Delta}\psi(t_n) \big)(\mathcal{I}_u(s) + \overline{\mathcal{I}_u(s)})ds.\label{J3}
\end{align}
{Recall that  $\mathcal{I}_u$ is given by \eqref{I_u} and $\mathcal{I}_{\psi}$ by \eqref{I_psi}.} We may now tackle the three highly oscillatory integrals separately.

\begin{lemma}[Second order approximation of the integral \eqref{J1}]\label{lemma:J1}
It holds
\begin{align*}
     J_1(w,v)&=\int_0^{\tau}e^{-i\frac12s\Delta}\big( e^{i\frac12s\Delta}v \big)\big( e^{isc\nab}w + e^{-isc\nab}\overline{w} \big)ds\\
&= \tau \big[ v\varphi_1(i\tau(c\nab-\tfrac{1}{2}\Delta))w  + v\varphi_1(-i\tau(c\nab+\tfrac{1}{2}\Delta))\overline{w}  \\
&+ e^{-i\frac12\tau\Delta}\big(e^{i\frac12\tau\Delta}v\big)\big[e^{i\frac12\tau\Delta}\Psi_2(i\tau(c\nab-\tfrac{1}{2}\Delta))w+ e^{i\frac12\tau\Delta}\Psi_2(-i\tau(c\nab+\tfrac{1}{2}\Delta))\overline{w}\big]\\
    &- v\Psi_2(i\tau(c\nab-\tfrac{1}{2}\Delta)){w} - v\Psi_2(-i\tau(c\nab+\tfrac{1}{2}\Delta))\overline{w} \big] + \mathcal{O}{(\tau^3 \mathcal{C}^2[f_{\text{quad}}(\cdot,\cdot),i\tfrac{1}{2}\Delta](v , w))}
    \\&=\tilde{J}_{1}(w,v) + \mathcal{O}{(\tau^3 \mathcal{C}^2[f_{\text{quad}}(\cdot,\cdot),i\tfrac{1}{2}\Delta](v , w))}.
\end{align*}
\end{lemma}
\begin{proof}
We define the following filter functions:
\begin{align*}
    &\mathcal{N}(s,s_1,v,w)=e^{-i\frac12s_1\Delta}\big( e^{i\frac12s_1\Delta} v \big)\big(  e^{i\frac12s_1\Delta}e^{-i\frac12s\Delta}[e^{isc\nab}w + e^{-isc\nab}\overline{w}] \big).
\end{align*}
Now, via a second order expansion of these two filter functions and a finite difference approximation of their first derivative with respect to $s_1$, we obtain the assertion.
\end{proof}

As for the integral \eqref{J2}, iterating Duhamel's formula yields
\begin{align*}
    J_2(u(t_n),\psi(t_n)) &=\int_0^{\tau}e^{-i\frac12s\Delta}\bigg( 
    \frac{i}{2} e^{i\frac12s\Delta}\int_0^{s}e^{-i\frac12\sigma \Delta}\big(e^{i\frac12 \sigma \Delta}\psi(t_n)\big)\big( e^{i \sigma c\nab}u(t_n) \big)d\sigma
    \bigg)\big( e^{isc\nab}u(t_n) \big)ds 
    \\&+\int_0^{\tau}e^{-i\frac12s\Delta}\bigg( 
    \frac{i}{2} e^{i\frac12s\Delta}\int_0^{s}e^{-i\frac12\sigma \Delta}\big(e^{i\frac12 \sigma \Delta}\psi(t_n)\big)\big( e^{i \sigma c\nab}u(t_n) \big)d\sigma
    \bigg)\big(e^{-isc\nab}\overline{u(t_n)} \big)ds
    \\&+\int_0^{\tau}e^{-i\frac12s\Delta}\bigg( 
    \frac{i}{2} e^{i\frac12s\Delta}\int_0^{s}e^{-i\frac12\sigma \Delta}\big(e^{i\frac12 \sigma \Delta}\psi(t_n)\big)\big( e^{-i \sigma c\nab}\overline{u(t_n)} \big)d\sigma
    \bigg)\big( e^{isc\nab}u(t_n) \big)ds
    \\&+\int_0^{\tau}e^{-i\frac12s\Delta}\bigg( 
    \frac{i}{2} e^{i\frac12s\Delta}\int_0^{s}e^{-i\frac12\sigma \Delta}\big(e^{i\frac12 \sigma \Delta}\psi(t_n)\big)\big( e^{-i \sigma c\nab}\overline{u(t_n)} \big)d\sigma
    \bigg)\big( e^{-isc\nab}\overline{u(t_n)} \big)ds \\
    &+ \mathcal{R}_4'
    \\&=J_{2,1}(u(t_n),\psi(t_n))+J_{2,2}(u(t_n),\psi(t_n))+J_{2,3}(u(t_n),\psi(t_n))+J_{2,4}(u(t_n),\psi(t_n)) \\
    &+ \mathcal{R}_4',
\end{align*}
where $\mathcal{R}_4'$ fulfills 
\begin{align}\label{R_4p}
    \|\mathcal{R}'_4\|_r \leq \tau^3  K\big(\sup_{0\leq \xi\leq \tau} \|u(t_n+\xi)\|_r,\sup_{0\leq \xi\leq \tau} \|\psi(t_n+\xi)\|_r\big).
\end{align}

We now handle the first term in $J_2$ in detail and the remaining three terms can be handled analogously. 
\begin{lemma}[Second order approximation of the integral \eqref{J21}]\label{lemma:J21}
For $0\leq\alpha\leq1$ it holds that
\begin{align*}
    J_{2,1}(w,v)&=\frac{\tau}{2c^2} \big[ vw\big( \varphi_1(i\tau(c^2+c\nab-\tfrac{1}{2}\Delta))-\varphi_1(i\tau(c\nab-\tfrac{1}{2}\Delta)) \big)w
    \big] \\
    &+ \mathcal{O}{(\tau^3( \Delta(vw)+c^{-2\alpha}\Delta^{1+\alpha}(vw)))}
    \\&=\tilde{J}_{2,1}(w,v) + \mathcal{O}{(\tau^3( \Delta(vw)+c^{-2\alpha}\Delta^{1+\alpha}(vw)))}.
\end{align*}
\end{lemma}
\begin{proof}
Proceeding as in the derivation of \eqref{Schro1} and using \eqref{anastasiya}, we obtain
\begin{align}\label{J21}
    J_{2,1}(w,v)&=\frac{i}{2}\int_0^{\tau}s\varphi_1(isc^2)e^{-i\frac12s\Delta}\big( 
    e^{i\frac12s\Delta} vw  \big)\big( e^{isc\nab}w\big)ds + \mathcal{R},
\end{align}
where $\mathcal{R}$ fulfills 
$$
\|\mathcal{R}\|_r \leq \tau^3 K\big( \Delta(vw) +c^{-2\alpha}\Delta^{1+\alpha}(vw)\big).
$$
Finally, the assertion follows with the definition of the filtered function
$$
\mathcal{N}(s,s_1,v,w) = e^{-i\frac12s_1\Delta}\big( 
    e^{i\frac12s_1\Delta} vw \big)\big( e^{i\frac12s_1\Delta} e^{-i\frac12s\Delta} e^{isc\nab}w \big),
$$
via a Taylor expansion up to order one.
\end{proof}

Analogously, we obtain that, for $0\leq\alpha\leq1$,
\begin{equation}\label{J22}
    \begin{aligned}
    J_{2,2}(w,v)&=\frac{\tau}{2c^2} \big[ vw\big( \varphi_1(i\tau(c^2-c\nab-\tfrac{1}{2}\Delta))-\varphi_1(-i\tau(c\nab+\tfrac{1}{2}\Delta)) \big)\overline{w}    \big] \\
    &+ \mathcal{O}{(\tau^3( \Delta(vw)+c^{-2\alpha}\Delta^{1+\alpha}(vw)))}
    \\&=\tilde{J}_{2,2}(w,v) + \mathcal{O}{(\tau^3( \Delta(vw)+c^{-2\alpha}\Delta^{1+\alpha}(vw)))},
    \end{aligned}
\end{equation}
\begin{equation}\label{J23}
    \begin{aligned}
    J_{2,3}(w,v)&=-\frac{\tau}{2c^2} \big[v\overline{w}\big( \varphi_1(i\tau(-c^2+c\nab-\tfrac{1}{2}\Delta))-\varphi_1(i\tau(c\nab-\tfrac{1}{2}\Delta)) \big)w    \big] \\
    &+ \mathcal{O}{(\tau^3(\mathcal{C}[f_{\text{quad}}(\cdot,\cdot),i\tfrac{1}{2}\Delta](v , w)+c^{-2\alpha}\Delta^{1+\alpha}(vw)))}
    \\&=\tilde{J}_{2,3}(w,v) + \mathcal{O}{(\tau^3(\mathcal{C}[f_{\text{quad}}(\cdot,\cdot),i\tfrac{1}{2}\Delta](v , w)+c^{-2\alpha}\Delta^{1+\alpha}(vw)))}
    \end{aligned}
\end{equation}
and
\begin{equation}\label{J24}
    \begin{aligned}
    J_{2,4}(w,v)&=-\frac{\tau}{2c^2} \big[ v\overline{w}\big( \varphi_1(-i\tau(c^2+c\nab+\tfrac{1}{2}\Delta))-\varphi_1(-i\tau(c\nab+\frac{1}{2}\Delta)) \big)\overline{w} \big] \\
    &+\mathcal{O}{(\tau^3(\mathcal{C}[f_{\text{quad}}(\cdot,\cdot),i\tfrac{1}{2}\Delta](v , w)+c^{-2\alpha}\Delta^{1+\alpha}(vw)))}
    \\&=\tilde{J}_{2,4}(w,v) + \mathcal{O}{(\tau^3(\mathcal{C}[f_{\text{quad}}(\cdot,\cdot),i\tfrac{1}{2}\Delta](v , w)+c^{-2\alpha}\Delta^{1+\alpha}(vw)))}.
    \end{aligned}
\end{equation}
Finally, we may approximate the last integral \eqref{J3} as follows. We iterate Duhamel's formula, obtaining
\begin{align*}
    J_3(u(t_n),\psi(t_n))&=\int_0^{\tau}e^{-i\frac12s\Delta}\big( e^{i\frac12s\Delta}\psi(t_n) \big)\bigg(  -ic\nab^{-1}e^{isc\nab}\int_0^se^{-i\sigma c\nab}\vert e^{i\frac12\sigma\Delta}\psi(t_n)\vert^2 d\sigma \bigg)ds 
    \\&+\int_0^{\tau}e^{-i\frac12s\Delta}\big( e^{i\frac12s\Delta}\psi(t_n)\big)\bigg(  ic\nab^{-1}e^{-isc\nab}\int_0^se^{i\sigma c\nab}\vert e^{i\frac12\sigma\Delta}\psi(t_n)\vert^2 d\sigma \bigg)ds + \mathcal{R}_4''
    \\&= J_{3,1}(u(t_n),\psi(t_n)) + J_{3,2}(u(t_n),\psi(t_n)) +\mathcal{R}_4'',
\end{align*}
where $\mathcal{R}_4''$ fulfills 
\begin{align}\label{R_4pp}
    \|\mathcal{R}_4''\|_r\leq \tau^3  K\big(\sup_{0\leq \xi\leq \tau} \|u(t_n+\xi)\|_r,\sup_{0\leq \xi\leq \tau} \|\psi(t_n+\xi)\|_r\big).
\end{align}

\begin{lemma}\label{lemma:J31}
For $0\leq\alpha\leq1$ it holds that
\begin{equation}
    \begin{aligned}
        J_{3,1}(w,v)&=\frac{\tau}{c^2}\big[ v\big( \varphi_1(i\tau(c\nab-c^2-\tfrac{1}{2}\Delta)) - \varphi_1(i\tau(c\nab-\tfrac{1}{2}\Delta)) \big)c\nab^{-1}\vert v\vert^2\big]+ \mathcal{O}{(\tau^3 w\Delta v)}
    \\&=\tilde{J}_{3,1}(w,v) + \mathcal{O}{(\tau^3 (\Delta v + c^{-2\alpha}\Delta^{1+\alpha}v ))}.
    \end{aligned}
\end{equation}
\end{lemma}
\begin{proof}
As for the first term, proceeding as in the derivation of \eqref{KGo1} and arguing similarly as in \eqref{anastasiya}, we obtain
\begin{align*}
    J_{3,1}(w,v)&=\int_0^{\tau}e^{-i\frac12s\Delta}\big( e^{i\frac12s\Delta}v \big)\big( -i s c \nab^{-1} e^{i s c\nab} v \varphi_1(-i s (c^2+\Delta))   \overline{v} \big)ds + \mathcal{R}_1 \\
    &=\int_0^{\tau}e^{-i\frac12s\Delta}\big( e^{i\frac12s\Delta}v \big)\big( -i s c \nab^{-1} e^{i s c\nab} v \varphi_1(-i s c^2)   \overline{v} \big)ds + \mathcal{R}_1+ \mathcal{R}_2,
\end{align*}
with
$$
\|\mathcal{R}_1\|_r\leq \tau^3 K\big( \mathcal{C}[f_{\text{quad}}(\cdot,\cdot),\Delta](v,v)+c^{-2\alpha}\Delta^{1+\alpha}v \big),
$$
by Lemma \ref{Lemma1ou} and
$$
\|\mathcal{R}_2\|_r\leq \tau^3 K\big( \Delta v \big),
$$
by \eqref{c_trick}. We define
$$
\mathcal{N}(s,s_1,v)=e^{-i\frac12s_1\Delta}\big(e^{i\frac12s_1\Delta}v\big)\big(e^{i\frac12s_1\Delta}e^{-i\frac12s\Delta}e^{isc\nab}\overline{v}\big)
$$
and obtain the assertion similarly as in Lemma \eqref{lemma:J21}, via a first order Taylor expansion of $\mathcal{N}(s,s_1,v)$ at $s_1=0$.
\end{proof}

Analogously we obtain the second term 
\begin{equation}\label{J32}
    \begin{aligned}
    J_{3,2}(w,v)&=\frac{\tau}{c^2} \big[ v\big( \varphi_1(i\tau(-c\nab+c^2-\tfrac{1}{2}\Delta)) - \varphi_1(-i\tau(c\nab+\tfrac{1}{2}\Delta)) \big)c\nab^{-1}\vert v\vert^2\big]+ \mathcal{O}{(\tau^3 w\Delta v)}
    \\&=\tilde{J}_{3,2}(w,v) + \mathcal{O}{(\tau^3 w\Delta v)}.
    \end{aligned}
\end{equation}
Collecting our results, namely Lemma \ref{lemma:J1}, Lemma \ref{lemma:J21}, \eqref{J22}, \eqref{J23}, \eqref{J24}, Lemma \ref{lemma:J31} and \eqref{J32}, leads to the following Corollary, where we use in addition the bound found in Lemma \ref{lemma_comm}.
\begin{cor}\label{cor:I_psi}
For $0\leq\alpha\leq1$ it holds that
 \begin{align*}
      \mathfrak{I}_{\psi}(w,v) 
     &= \tilde{J}_1(w,v) + \tilde{J}_{2,1}(w,v) + \tilde{J}_{2,2}(w,v)+\tilde{J}_{2,3}(w,v) +\tilde{J}_{2,4}(w,v)+\tilde{J}_{3,1}(w,v)+\tilde{J}_{3,2}(w,v)\\
     &+\mathcal{O}{(\tau^3(\Delta w v +c^{-2\alpha}\Delta^{1+\alpha}vw))}\\
     &= \Tilde{\mathfrak{I}}_{\psi}(w,v) + \mathcal{O}{(\tau^3(\Delta w v +c^{-2\alpha}\Delta^{1+\alpha}vw))}.
 \end{align*}
\end{cor}
These considerations collected in Corollary \ref{cor:I_psi} lead to the second order uniformly accurate low regularity integrator:
\begin{equation}\label{Schro2}
    \begin{aligned}
    \psi^{n+1} &= e^{i\frac12\tau\Delta}\psi^n + \frac{i}{2}e^{i\frac12\tau\Delta}\tilde{\mathfrak{I}}_{\psi}(u^n,\psi^n).
    \end{aligned}
\end{equation}

% --------------------------------------------------------------------------------------- modify for 2nd order !!!!
In the sections that follow we aim to carry out the error analysis of the scheme in $(u^{n},\psi^{n})$ given by \eqref{KGo2} and \eqref{Schro2}. We recall that we denote by $\varphi^t_{K}$, $\varphi^t_{S}$ the exact flows of \eqref{KG} and \eqref{Schr} respectively and we let $\tilde{\Phi}^t_{K}$, $\tilde{\Phi}^t_{S}$ be the numerical flows corresponding to \eqref{KGo2} and \eqref{Schro2} respectively, such that in particular it holds
$$
u^{n+1}=\tilde{\Phi}^{\tau}_{K}(u^n,\psi^n),\quad \psi^{n+1}=\tilde{\Phi}^{\tau}_{S}(u^n,\psi^n).
$$
\subsection{Local error analysis}
\begin{lemma}
\label{localerror2}
Fix $r>\frac{d}{2}$. The local error given by the differences $\varphi^{\tau}_K(u(t_n),\psi(t_n))-\tilde{\Phi}^{\tau}_K(u(t_n),\psi(t_n))$ and $\varphi^{\tau}_S(u(t_n),\psi(t_n))-\tilde{\Phi}^{\tau}_S(u(t_n),\psi(t_n))$ satisfies
$$
\varphi^{\tau}_K(u(t_n),\psi(t_n))-\tilde{\Phi}^{\tau}_K(u(t_n),\psi(t_n)) = \mathcal{O}\big(\tau^3(\Delta u(t_n){\psi(t_n)}+ c^{-2\alpha}\Delta^{1+\alpha}u(t_n)\psi(t_n) + c^{-4\alpha}\Delta^{2+2\alpha}\partial_x\psi(t_n))\big)
$$
and
$$
\varphi^{\tau}_S(u(t_n),\psi(t_n))-\tilde{\Phi}^{\tau}_S(u(t_n),\psi(t_n)) = \mathcal{O}\big(\tau^3(\Delta u(t_n){\psi(t_n)} + c^{-2\alpha}\Delta^{1+\alpha}u(t_n)\psi(t_n) )\big),
$$
where $0\leq\alpha\leq1$.
\end{lemma}
\begin{proof}
This assertion follows by Corollary \ref{cor:I_u}, \eqref{2o_r1} and Corollary \ref{cor:I_psi}, \eqref{R_4}, \eqref{R_4p} and \eqref{R_4pp} respectively.
\end{proof}
\subsection{Stability analysis}
\begin{lemma}
\label{stability2}
Fix $r>\frac{d}{2}$. The numerical flows $\tilde{\Phi}^{\tau}_K$ and $\tilde{\Phi}^{\tau}_S$ defined by \eqref{KGo2} and \eqref{Schro2} respectively are stable in $H^r$, namely it holds for any $v_i,w_i\in H^r$, $i\in\{1,2\}$ that
$$
\|\tilde{\Phi}^{\tau}_K(v_1,w_1)-\tilde{\Phi}^{\tau}_K(v_2,w_2)\|_r\leq \|v_1-v_2\|_r + \tau M_1 (\|v_1-v_2\|_r +\|w_1-w_2\|_r),
$$
$$
\|\tilde{\Phi}^{\tau}_S(v_1,w_1)-\tilde{\Phi}^{\tau}_S(v_2,w_2)\|_r\leq \|w_1-w_2\|_r + \tau M_2 (\|v_1-v_2\|_r +\|w_1-w_2\|_r),
$$
where $M_1$ and $M_2$ can be chosen independently of $c$.
\end{lemma}
\begin{proof}
This claim follows by \eqref{cnab_liniso_bound}. In addition, we use the fact that it holds $\vert \varphi_1 (i\xi)\vert\leq 1$ for all $\xi\in\mathbb{R}$.
\end{proof}
\subsection{Global error}
\begin{theorem}\label{thm:2}
Fix $r>\frac{d}{2}$ and assume that the solution $(u,\psi)$ of \eqref{KG}-\eqref{Schr} satisfies $u\in\mathcal{C}([0,T],H^{r+2+2\alpha})$ and $\psi\in\mathcal{C}([0,T],H^{r+5+4\alpha})$, $0\leq\alpha\leq1$. Then there exists a $\tau_0>0$ such that for all $0<\tau\leq\tau_0$ the following estimate holds for $(u^n,\psi^n)$ defined in \eqref{KGo2} and \eqref{Schro2}
\begin{align*}
    \|u(t_n)-u^n\|_r  &\leq \tau^2 K\big(\sup_{t_n\leq t\leq t_{n+1}} \|u(t)\|_{r+2},\sup_{t_n\leq t\leq t_{n+1}} \|\psi(t)\|_{r+2}\big)\\
    & +  \tau^2 c^{-2\alpha} K\big(\sup_{t_n\leq t\leq t_{n+1}} \|u(t)\|_{r+2\alpha+2},\sup_{t_n\leq t\leq t_{n+1}} \|\psi(t)\|_{r+2\alpha+2}\big)\\
    &+\tau^2 c^{-4\alpha} K\big(\sup_{t_n\leq t\leq t_{n+1}} \|\psi(t)\|_{r+4\alpha+5}\big),
\end{align*}
and
\begin{align*}
    \|\psi(t_n)-\psi^n\|_r &\leq \tau^2 K\big(\sup_{t_n\leq t\leq t_{n+1}} \|u(t)\|_{r+2},\sup_{t_n\leq t\leq t_{n+1}} \|\psi(t)\|_{r+2}\big)\\
    &+\tau^2 c^{-2\alpha}K\big(\sup_{t_n\leq t\leq t_{n+1}} \|u(t)\|_{r+2+2\alpha},\sup_{t_n\leq t\leq t_{n+1}} \|\psi(t)\|_{r+2+2\alpha}\big)
\end{align*}
where $0\leq\alpha\leq1$ and, in particular, $K$ can be chosen independently of $c$.
\end{theorem}
\begin{proof}
The proof follows by means of a Lady Windermere's fan argument, after plugging in the results obtained in Lemmata \ref{localerror2} and \ref{stability2} and using the regularity estimates for the commutator terms obtained in Lemma \ref{lemma_comm}.
\end{proof}

\section{Asymptotic consistency}
In this section we show that our novel class of first and second order integrators are asymptotically consistent, meaning that in the limit $c\to\infty$ we recover the solution of the limit system. 

The limit system can be for instance derived via Modulated Fourier Expansion techniques, see for example \cite{CHL}, \cite{FaS}, \cite{HL}, \cite{HLW}. We refer to \cite{BaKoS18} for the details of this derivation. 

% Indeed, as $c\to\infty$, for sufficiently smooth initial data, the solution to \eqref{KGS} admits an expansion of the form
%$$
%z(t,x) = \frac12 \bigg( e^{ic^2t} u_{\infty}(t,x) + e^{-ic^2t} \overline{u_{\infty}(t,x)} \bigg) +\mathcal{O}(c^{-2}) = z_{\infty}(t,x) +\mathcal{O}(c^{-2})
%$$
%where $(u_{\infty},\psi_{\infty})$ solve the following free Schr\"odinger limit system.
%\begin{equation}\label{limitsys}
%    \begin{aligned}
%    &\partial_t u_{\infty}(t,x)=-\tfrac{i}{2}\Delta u_{\infty}(t,x),\quad &&u_{\infty}(0)=z(0) - %ic^{-2}\partial_t z(0),\\
%    &\partial_t\psi_{\infty}=i\Delta\psi_{\infty},\quad &&\psi_{\infty}(0) = \psi_0.
%    \end{aligned}
%\end{equation}\todo{$\frac12$}
%Note that this is a decoupled linear system, thus it can be solved exactly in time, however we will express its solution in terms of a numerical integration scheme that reads as follows
%\begin{equation}\label{sol_limitsys}
%    \begin{aligned}
%    &u_{\infty}^{n+1} = e^{-\frac12 i\tau \Delta}u_{\infty}^n,\quad&&u_{\infty}^0=z(0) - ic^{-2}\partial_t z(0),\\
%    &\psi_{\infty}^{n+1}=e^{i\tau\Delta}\psi_{\infty}^n,\quad &&\psi_{\infty}^0 = \psi_0.
%    \end{aligned}
%\end{equation}
%We then easily recover $z_{\infty}^{n+1}$ by
%$$
%z_{\infty}^{n+1} = \frac12 \bigg( e^{ic^2t_{n+1}} u_{\infty}^{n+1} + e^{-ic^2t_{n+1}} \overline{u_{\infty}^{n+1}} \bigg). 
%$$
\subsection{Asymptotic convergence of the first order method}
In this section we motivate why the method given by \eqref{KGo1} and \eqref{Schro1} converges towards the solution of the limit system as $c\to\infty$.

We see that formally it holds
\begin{align}\label{ac_1}
    e^{i\tau c\nab} = e^{i\tau (c^2+\frac12\Delta)} +\mathcal{O}(c^{-2}),
\end{align}
as Taylor series expansion of the function $x\mapsto c\sqrt{x+c^2}$ around the point zero shows
\begin{align}\label{h4req}
    \big\|c\nab f - \big(c^2-\frac12 \Delta\big)f\big\|_r\leq K c^{-2}\|f\|_{r+4},
\end{align}
for some $K>0$ independent of $c$. Note that this particular asymptotic bound requires additional regularity for $u$.

It follows by \eqref{cnab_liniso_bound}, the observation 
\begin{align}\label{ac_2}
    \big\|\tau \varphi_1(\pm i\tau c^2) \big\|_r \leq \frac{2}{c^2},
\end{align}
and by \eqref{anastasiya}, that for $u^n$ given by \eqref{KGo1} it holds
$$
u^{n+1} = e^{-\frac12 i\tau \Delta}u^n + \mathcal{O}(c^{-2}).
$$
As for $\psi^n$ given by \eqref{Schro1}, we see by \eqref{ac_2} that 
$$
\psi^{n+1} = e^{i\tau\Delta}\psi^n + \mathcal{O}(c^{-2}).
$$

\subsection{Asymptotic convergence of the second order method}
Analogously to the previous section, using \eqref{cnab_liniso_bound}, \eqref{ac_1} and \eqref{ac_2} together with
\begin{align}\label{ac_3}
    \big\|\tau\Psi_2(i\tau c^2)\big\|_r\leq \frac{2}{c^2}
\end{align}
and \eqref{anastasiya}, we are able to see that indeed $\tilde{\mathfrak{I}}_u$ given in Corollary \ref{cor:I_u} fulfills
$$
ic\nab^{-1}\tilde{\mathfrak{I}}_u(u^n,\psi^n) = \mathcal{O}(c^{-2}),
$$
Note that, by \eqref{h4req}, we also require here $H^4$ for $u$. Similarly, with \eqref{ac_2} and \eqref{ac_3}, we obtain that $\tilde{\mathfrak{I}}_{\psi}$ given in Corollary \ref{cor:I_psi} fulfills
$$
\tilde{\mathfrak{I}}_{\psi}(u^n,\psi^n) = \mathcal{O}(c^{-2}).
$$
This implies that the second order method given by \eqref{KGo2} and \eqref{Schro2} also converges to the corresponding solution of the limit system with order $c^{-2}$ formally.

\section{Numerical Experiments}
{We dedicate this last section to the numerical verification of our results. We mainly concentrate on the convergence of the first order method in order to illustrate the explicit relation in our error estimates between gain in $c^{-2\alpha}$, $0\leq\alpha\leq1$, for large $c$, and consequent loss in derivative. In particular, we observe uniform accuracy and and improvement in convergence for more regular initial data and large $c$, as depicted in Theorem \ref{thm:1}. Then we briefly present the convergence results for the second order method, which verify second order convergence and uniform accuracy, as obtained in Theorem \ref{thm:2}. We leave out the experiments for different regularity assumptions in this case for the sake of brevity. For the spatial discretisation we use a standard Fourier pseudospectral method, choosing $M=200$ as the highest Fourier mode. See \cite{STW} for more information on this technique as well as for some applications.}

In Figure 1 we plot the global error of the first order scheme given by \eqref{KGo1} and \eqref{Schro1} measured in $H^1$ for different values of $c$ and initial data in Sobolev spaces varying in regularity, as well as the convergence of the second order scheme given by \eqref{KGo2} and \eqref{Schro2} measured in $H^1$ for different values of $c$ and smooth initial data.  %Our observations point to a behaviour of  $\min\{ \tau,c^{-2} \}$ and $\min\{ \tau^2,c^{-2} \}$ respectively of the global error, something that is not obvious when we look at the global error bounds found in Theorems \ref{thm:1} and \ref{thm:2} respectively.

%In Figure \ref{initial_data_plots} we respectively plot the $H^1$ and $H^2$ initial data employed to generate the plots in Figure \ref{order_plot}.  

%Finally, in Figure \ref{asymptotic_cons_plots} we illustrate the asymptotic convergence of the first and second order schemes respectively. More precisely, for different values of $c$, we plot the difference between the numerical solutions and the solution of the limit system measured in the $L^2$-norm. In this last numerical experiment, we have used a time step size of $\tau = 2^{-9}$ and the initial data
$$
%z(0,x) = \frac12\frac{\cos^2(x)}{2-\cos(x)},\quad \partial_t z(0,x) =\frac12\frac{\sin(x)\cos(x)}{2-\cos(x)} ,\quad %\psi(0,x)=1+\frac{i\sin(x)}{2-\cos(x)}.
$$
%Furthermore, we have carried out some experiments generated with $H^i$ initial data, $i\in\{ 1,2,3,4 \}$. In particular, we see the necessity of $H^4$ initial data for the Klein--Gordon component in order to achieve order $c^{-2}$ asymptotic convergence. Nonetheless, our numerical experiments hint at asymptotic convergence for less regular data albeit with a slower rate.
 \begin{figure}[h!]
  \includegraphics[width=0.95\textwidth]{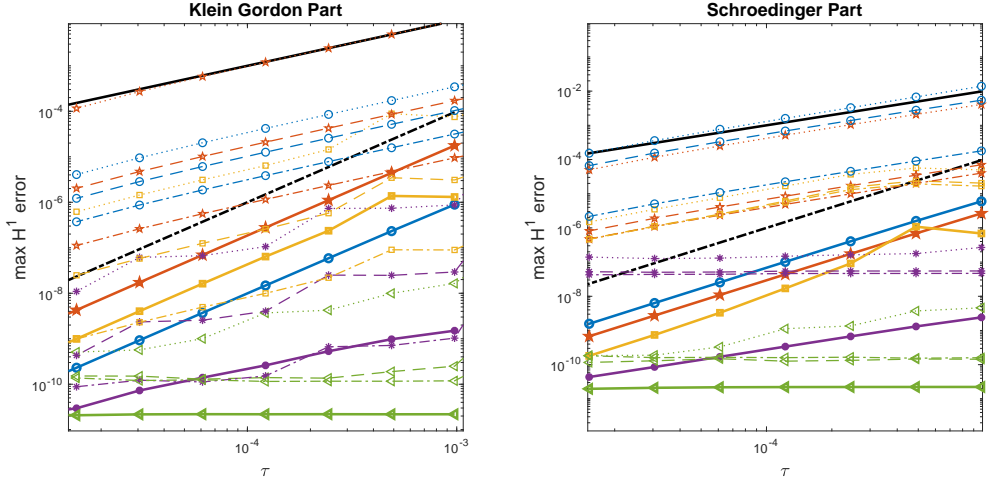}
  \caption{Convergence plot of the first and second order scheme given by \eqref{KGo1}, \eqref{Schro1} and \eqref{KGo2}, \eqref{Schro2}. The blue, orange, yellow, purple and green lines correspond to the values of $c=1$, $c=10$, $c=100$, $c=1000$ and $c=10000$ respectively. The black thick lines are reference lines of slope one and two. The ticker solid lines correspond to the second order scheme with smooth initial data. The thinner dotted, dashed and mixed lines correspond to the first order scheme with $H^2$, $H^3$ and $H^4$ initial data respectively.}
  \label{order_plot}
\end{figure}

 %\begin{figure}[h!]
 % \includegraphics[width=0.45\textwidth]{h1data}
 % \includegraphics[width=0.45\textwidth]{h2data}
 % \caption{$H^1$ and $H^2$ initial data used to generate the convergence plots in Figure 1 and 2 respectively.}
 % \label{initial_data_plots}
% \end{figure}

 %\begin{figure}[h!]
 % \includegraphics[width=0.47\textwidth]{asymptotic_cons_1_mult}
 % \includegraphics[width=0.47\textwidth]{asymptotic_cons_2_mult}
 % \caption{Asymptotic consistency plots.}
 % \label{asymptotic_cons_plots}
%\end{figure}

\newpage

\subsection*{Acknowledgements}

{\small
The author has received funding from the European Research Council (ERC) under the European Union’s Horizon 2020 research and innovation programme (grant agreement No. 850941).
}

\end{document}